\documentclass[11pt]{amsart}
\usepackage{amssymb,amsmath}

\makeatletter
\theoremstyle{plain}
 \newtheorem{thm}{Theorem}[section]
\newtheorem{thm*}{Theorem}
 
 \newtheorem{prop}[thm]{Proposition}
 
 \numberwithin{equation}{section} 
\numberwithin{figure}{section} 
 \theoremstyle{plain}
 \theoremstyle{definition}
 \newtheorem{defn}[thm]{Definition}
 \newtheorem{rem}[thm]{Remark}

\newcommand{\calA}{{{\mathcal A}}}
\newcommand{\fC}{{{\mathfrak C}}}

\newcommand{\calS}{{\mathcal{S}}}
\newcommand{\calB}{{{\mathcal B}}}

\newcommand{\calL}{{{\mathcal L}}}

\newcommand{\C}{{{\mathbb C}}}

\newcommand{\R}{{{\mathbb R}}}

\newcommand{\bc}{{{\bf c}}}
\newcommand{\bp}{{{\bf p}}}

\newcommand{\bu}{{{\bf u}}}
\newcommand{\cC}{{{\mathcal C}}}
\newcommand{\cm}{{{\mathfrak m}}}

\newcommand{\h}{{{\mathfrak H}}}

\makeatother

\begin{document}

\title[Revolution rings]{Modulus of revolution rings in the heisenberg group}

\author[I.D. Platis ]{Ioannis D. Platis}

\email{jplatis@math.uoc.gr}
\address{Department of Mathematics and Applied Mathematics\\
 University of Crete
\\ University Campus\\
GR 700 13 Voutes Heraklion Crete\\Greece}

\begin{abstract}
Let $\calS$ be a surface of revolution embedded in the Heisenberg group $\h$. A revolution ring $R_{a,b}(\calS)$, $0<a<b$, is a domain in  $\h$ bounded by two dilated images of $\calS$,  with dilation factors $a$ and $b$, respectively.
We prove that if $\calS$ is subject to certain geometric conditions, then the modulus of the family $\Gamma$ of horizontal boundary connecting curves inside $R_{a,b}(\calS)$ is
$$
{\rm Mod}(\Gamma)=\pi^2(\log(b/a))^{-3}.
$$
Our result applies for many interesting surfaces, e.g., the Kor\'anyi metric sphere, the Carnot-Carath\'eodory metric sphere and the bubble set.
\end{abstract}

\maketitle
\section{Introduction}
The Heisenberg group $\h$ is the set $\C\times\R$ with multiplication $*$ given by
$$
(z,t)*(w,s)=(z+w,t+s+2\Im(z\overline{w})),
$$
for every $(z,t)$ and $(w,s)$ in $\h$. With this multiplication $\h$ becomes a 2-step nilpotent Lie group and constitutes the primary model for sub-Riemannian geometry. Besides the horizontal (Carnot-Carath\'eodory) metric $d_{cc}$, there is another metric $d_\h$ defined in $\h$, which is called the Kor\'anyi metric. This is not a path metric; nevertheless it is bi-Lipschitz equivalent to $d_{cc}$. Let $\Omega\subset\h$ be a domain and $\Gamma$ be a family of smooth rectifiable curves with respect to the horizontal metric, lying in $\Omega$. Then the {\it modulus of $\Gamma$} is defined by
$$
{\rm Mod}(\Gamma)=\inf_{\rho\in {\rm Adm}(\Gamma)}\iiint_\h \rho^4 d\cm^3.
$$
Here, $d\cm^3$ is the Lebesgue measure in $\C\times\R$ and ${\rm Adm}(\Gamma)$ is the set of positive Borel functions of $\Omega$ which are such that their horizontal line integral with respect to each $\gamma\in\Gamma$ is greater or equal to 1, see Section \ref{sec:mod} for more details. The function $\rho_0$ which realises the infimum above is called an extremal density for $\Gamma$. Moduli of curve families inside domains of $\h$ are quite important in the theory of quasiconformal mappings of the Heisenberg group, see for instance \cite{KR2}. Calculation of moduli of curve families and their respective extremal densities in domains of $\h$ lead to moduli methods which have been proved a quite powerful tool for the solution of extremal problems, see for instance \cite{BFP}. This idea goes back to Gr\"otzsch and the initiation of study of extremal problems in the complex plane. But in contrast to the complex plane situation and due to the lack of a Riemann mapping theorem, there can be no neat 
normalisation for an arbitrary domain in $\h$; therefore, to calculate the modulus of a curve family $\Gamma$ inside a domain $\Omega\subset\h$ is not in principle an easy task. In their paper \cite{KR}, Kor\'anyi and Reimann considered the ring $R_{a,b}$, $0<a<b$; this is the domain between two $d_\h$-metric spheres of radii $a$ and $b$ respectively. Using a particular set of coordinates for $\h$ they showed that the modulus of the family of boundary connecting curves, otherwise known as the {\it conformal capacity of $R_{a,b}$}, is 
$$
{\rm Mod}(\Gamma)=\pi^2(\log(b/a))^{-3},
$$
with extremal density 
$$\rho_0(z,t)=\frac{1}{\log(b/a))}\frac{|z|}{(|z|^4+t^2)^{1/2}}\mathcal{X}(R_{a,b}).$$ Here, $\mathcal{X}(R_{a,b})$ is the characteristic function of $R_{a,b}$. 

In this paper we consider a {\it revolution ring} $R_{a,b}(\calS)$. This is constructed as follows: Consider a $\cC^2$ curve $\bp$ lying on the $xt$-plane such that its closure has endpoints lying on the $t$-axis. We also impose sume plausible geometric conditions on $\bp$, see Conditions (\ref{A1}) and (\ref{A2}) in Section \ref{sec:surfrevg} and Condition (\ref{eq:betacond}) in Section \ref{sec:revring}. Let $\calS$ be the $\cC^2$ surface of revolution of $\bp$ around the $t$-axis and denote by $\overline{\calS}$ the compact set comprising $\calS$ and the boundary points of $\bp$. Then $R_{a,b}(\calS)$ is a domain in $\h$ bounded by two dilated images $D_a(\overline{\calS})$ and $D_b(\overline{\calS})$, $0<a<b$. Denote by $\Gamma$ the family of horizontal curves which lie inside $R_{a,b}(\calS)$ and connect its boundary components and let also ${\rm Mod}(\Gamma)$ be the modulus of $\Gamma$. 
Then according to (1) of Theorem \ref{thm:main1} we have that 
$$
{\rm Mod}(\Gamma)\ge \pi^2(\log(b/a))^{-3}=\iiint_\h\rho_0^4(z,t)d\cm^3(z,t),
$$
where
$$
\rho_0(z,t)=\frac{1}{\log(b/a))}\frac{|z|}{(|z|^4+t^2)^{1/2}}\mathcal{X}(R_{a,b}(\calS)).
$$
Moreover, from (2) of Theorem \ref{thm:main1} we have $${\rm Mod}(\Gamma)= \pi^2(\log(b/a))^{-3},$$
and thus $\rho_0$ is an extremal density for $\Gamma$.

We finally remark that to prove Theorem \ref{thm:main1} we extensively use {\it revolution coordinates} for $\h$ which we define on Section \ref{sec:revring}. These generalise the logarithmic coordinates for $\h$, see \cite{BFP}. 

The paper is organised as follows. In Section \ref{sec:prel} we review some known facts about the Heisenberg group. In Section \ref{sec:surfrev} we deal with $\cC^2$ surfaces of revolution embedded in $\h$ and their horizontal geometry. In Section \ref{sec:revring} we define revolution coordinates; finally, in Section \ref{sec:mod1} we prove Theorem \ref{thm:main1}.
\section{Preliminaries}\label{sec:prel}
Most of the material included in this section is well known. In Section \ref{sec:heis} we review in brief the Kor\'anyi-Cygan metric structure  of the Heisenberg group. In Section \ref{sec:con-leg} we describe its Lie group structure, its contact and its sub-Riemannian geometry. For further details we refer the reader for instance to \cite{CDPT}. The definition of modulus of families of rectifiable curves inside a domain of the Heisenberg group is in Section \ref{sec:mod}. 
Finally, a brief review of $\cC^2$ surfaces embedded in the Heisenberg group is given in Section \ref{sec:surf}; for this, we follow \cite{P}.
\subsection{Heisenberg group}\label{sec:heis}
Let $\h$ be the Heisenberg group as defined in the introduction.
The Kor\'anyi map $\alpha:\h\to \C$ is given by
$$
\alpha(z,t)=-|z|^2+it.
$$
From this we deduce the Kor\'anyi gauge $|\cdot|_\h$, which is defined by 
$$
\left|(z,t)\right|_\h=|\alpha(z,t)|^{1/2}=\left| |z|^2-it\right|^{1/2},
$$
for every $(z,t)\in\h$. The {\it Kor\'anyi--Cygan}  metric $d_\h$ is then defined by the relation
$$
d_\h\left((z_1,t_1),\,(z_2,t_2)\right)
=\left|(z_1,t_1)^{-1}*(z_2,t_2)\right|.
$$
The $d_\h$-sphere of radius $R>0$ and centred at the origin is the {\it Kor\'anyi sphere}
$$
\calS_\h(R)=\{(z,t)\in\h\;|\;\left|(z,t)\right|_\h=R\}.
$$
The metric $d_\h$ is invariant under {\it left translations} and {\it rotations around the vertical axis} $\mathcal{V}=\{0\}\times\R$. Left translations $T_{(\zeta,s)}$, $(\zeta,s)\in\h$, are defined by
 $
 T_{(\zeta,s)}(z,t)=(\zeta,s)*(z,t),
 $
 and rotations $R_\theta$, $\theta\in\R$, about the vertical axis $\mathcal{V}$ are defined by 
 $
 R_\theta(z,t)=(ze^{i\theta},t),
 $
 for every $(z,t)\in\h$.
Left translations are left actions of $\h$ onto itself and rotations are induced by an action of ${\rm U}(1)$ on $\h$; together they form the group ${\rm Isom}^+(\h,d_\h)$ of (orientation-preserving) {\it Heisenberg isometries}. Any other isometry of $d_\h$ is the composition of an element  of  ${\rm Isom}^+(\h,d_\h)$ with conjugation $j:(z,t)\mapsto(\overline{z},-t)$.
We also consider two other kinds of transformations, namely {\it dilations} and {\it inversion}. Dilations $ D_\delta$, $\delta>0$, are defined by
 $
 D_\delta(z,t)=(\delta z, \delta^2 t),
 $
for every $(z,t)\in\h$. One may show that 
the metric $d_\h$ is  scaled up to multiplicative constants by the action of  dilations. 
Finally, inversion $I$ is defined in $\h\setminus\{(0,0)\}$ by 
$
I(z,t)=\left(z\left(\alpha(z,t)\right)^{-1},
-t\left|\alpha(z,t)\right|^{-2}\right).
$
Compositions of orientation-preserving Heisenberg isometries, dilations and inversion form the group  ${\rm Sim}^+(\h)$ of orientation-preserving {\it similarities} of $\h$. The  {\it similarity group} ${\rm Sim}(\h)$ of $\h$ is the group comprising elements of ${\rm Sim}^+(\h)$ followed by conjugation $j$.  

\subsection{Contact and sub-Riemannian geometry of $\h$}\label{sec:con-leg}
The Heisenberg group $\h$ is a (two-step nilpotent) Lie group with underlying manifold $\R^2\times\R$. Consider the left invariant vector fields
\begin{eqnarray*}
X=\frac{\partial}{\partial x}+2y\frac{\partial}{\partial t},\quad Y=\frac{\partial}{\partial y}-2x\frac{\partial}{\partial t},
\quad T=\frac{\partial}{\partial t}.
\end{eqnarray*}
These form a basis or the Lie algebra of left invariant vector fields of $\h$. Moreover, $\mathfrak{h}$ has a grading $\mathfrak{h} = \mathfrak{v}_1\oplus \mathfrak{v}_2$ with
\begin{displaymath}
\mathfrak{v}_1 = \mathrm{span}_{\R}\{X, Y\}\quad \text{and}\quad \mathfrak{v}_2=\mathrm{span}_{\R}\{T\}.
\end{displaymath}
The {\it contact structure}  of $\h$ is induced by a 1-form $\omega$ of $\h$; this is defined as the unique 1-form which is such that $X,Y\in{\rm ker}\omega$, $\omega(T)=1$. In Heisenberg coordinates $z=x+iy,t,$ the contact form $\omega$ is given by
\begin{eqnarray*}
\omega=dt+2(xdy-ydx)=dt+2\Im(\overline{z}dz).
\end{eqnarray*}
Uniqueness of $\omega$ (modulo change of coordinates) follows by the contact version of Darboux's Theorem. The {\it sub-Riemannian geometry} of $\h$ is given by the distribution which is defined by the first layer $\mathfrak{v}_1$; this is the {\it horizontal distribution}. At each point $p\in\h$, $(\mathfrak{v}_1)_p={\rm H}_p(\h)$ is the horizontal tangent space of $\h$ at $p$. The sub-Riemannian metric $\langle\cdot,\cdot\rangle$ is given in the horizontal bundle by the relations
\begin{equation*}
\langle X,X\rangle=\langle Y,Y\rangle=1,\quad \langle X,Y\rangle=\langle Y,X\rangle=0,
\end{equation*}
and the induced norm shall be denoted by $\|\cdot\|$.
An absolutely continuous curve $\gamma:[a,b]\to \h$ (in the Euclidean sense) with 
\begin{displaymath}
\gamma(\tau)=(\gamma_h(\tau),\gamma_3(\tau))\in\mathbb{C}\times \mathbb{R},
\end{displaymath}
 is called {\it horizontal} if $\dot{\gamma}(\tau)\in {\rm H}_{\gamma(\tau)}(\h)$ for almost all $\tau\in [a,b]$; equivalently,
 \begin{displaymath}
 \dot t(\tau)=-2\Im\left(\overline{z(\tau)}\dot z(\tau)\right),
\end{displaymath}
for almost all $\tau\in [a,b]$.
A curve $\gamma:[a,b]\to \h$ is absolutely continuous with respect to $d_\h$ if and only if it is a horizontal curve. Moreover,
the horizontal length of a smooth rectifiable curve $\gamma=(\gamma_h,\gamma_3)$ with respect to $\|\cdot\|$ is given by the integral over the (Euclidean) norm of the horizontal part of the tangent vector,
\begin{displaymath}
 \ell_h(\gamma)=\int_a^b \|\dot{\gamma}_h(\tau)\|\;d\tau=\int_a^b\left(\langle\dot\gamma(\tau),X_{\gamma(\tau)}\rangle^2+\langle\dot\gamma(\tau),Y_{\gamma(\tau)}\rangle^2\right)^{1/2}d\tau.
\end{displaymath}
Thus the {\it Carnot-Carath\'eodory distance} of two arbitrary points $p,q\in\h$ is
$$
d_{cc}(p,q)=\inf_\gamma\ell_h(\gamma),
$$ 
where $\gamma$ is horizontal and joins $p$ and $q$. The $d_{cc}$-sphere of radius $R$ and centred at the origin is the {\it Carnot-Carath\'eodory sphere} $\calS_{cc}(R)$ and can be constructed as follows. Consider the family ${\bf c}_k$, $k\in\R$, of planar circles 
$$
{\bf c}_k(s)=\frac{1}{k}(1-e^{iks}),\quad s\in[0,2\pi/|k|].
$$
In the case $k=0$, circles are degenerated into straight line segments. For every $k$, ${\bf c}_k$ is lifted to the horizontal curve $\bp_0^k$,
$$
\bp_0^k(s)=\left(\bc_k(s),t_k(s)\right),\quad t_k(s)=\frac{2}{k}\left(\frac{1}{k}\sin(ks)-s 
\right).
$$
Denote by $\bp_\phi^k$ the rotation of $\bp_0^k$ around the vertical axis $\mathcal{V}$. Then
$$
\calS_{cc}(R)=\bp_\phi^k(R),\;k\in[-2\pi/R,2\pi/R],\;\phi\in[0,2\pi].
$$
\subsection{Modulus of curve families}\label{sec:mod}
Let $\Gamma$ be a family of rectifiable curves which lie inside a domain $\Omega$ of $\h$. A positive Borel function $\rho$ defined in $\Omega$ is called {\it admissible} for $\gamma$ if
$$
\int_\gamma\rho ds^h\ge 1,
$$
for all $\gamma\in\Gamma$. The curve integral on the left is defined as follows: if $\gamma(\tau)=\left(\gamma_h(\tau),\gamma_3(\tau)\right)$, $\tau\in [a,b]$, then
$$
\int_\gamma\rho ds^h=\int_a^b\rho\left(\gamma(\tau)\right)\|\dot\gamma_h(\tau)\|d\tau.
$$
Denote by ${\rm Adm}(\Gamma)$ the set of all admissible functions for $\Gamma$. Then the modulus ${\rm Mod}(\Gamma)$ of $\Gamma$ is
$$
{\rm Mod}(\Gamma)=\inf_{\rho\in {\rm Adm}(\Gamma)}\iiint_\h \rho^4 d\cm^3,
$$
where $d\cm^3$ is the Lebesgue measure in $\C\times\R$. Modulus ${\rm Mod}(\Gamma)$ is invariant by the action of the similarity group ${\rm Sim}(\h)$. By Theorem 8 in \cite{KR1}, the elements of  ${\rm Sim}(\h)$ are exactly the 1-quasiconformal self-mappings (both orientation-preserving and orientation-reversing) of $\h$. Therefore if $g\in{\rm Sim}(\h)$ then ${\rm Mod}\left(g(\Gamma)\right)=$ ${\rm Mod}(\Gamma)$, where $g(\Gamma)$ is the image curve family of horizontal curves lying inside $\Omega'=g(\Omega)$. This result is a consequence of the Modulus Inequality which holds for quasiconformal mappings of $\h$; for further details see for instance \cite{BFP} and the references therein.

\subsection{$\cC^2$ surfaces in the Heisenberg group}\label{sec:surf}
We list below some basic features of $\cC^2$ surfaces in the Heisenberg group. Such a surface is primarily a regular surface in $\R^3$, that is, a countable collection of surface patches $\sigma_\alpha:U_\alpha\to V_\alpha$, where $U_\alpha$ and $V_\alpha$ are open sets of $\R^2$ and $\R^3$, respectively, such that:
\begin{enumerate}
\item Each $\sigma_\alpha$ is a $\cC^2$ homeomorphism.
\item The differential $(\sigma_\alpha)_*$ is everywhere of rank 2.
\end{enumerate}
Let $p$ be an arbitrary point of $\calS$ and suppose that $\sigma:U\to\R^3$,
$$
\sigma(u,v)=\left(x(u,v),y(u,v),t(u,v)\right),
$$
is a surface patch of $\calS$ such that $p=\sigma(\bu_0)$ for some $\bu_0=(u_0,v_0)\in U$. The tangent plane $T_p(\calS)$ of $\calS$ at $p$ is spanned by the vectors
\begin{eqnarray*}
&&
(\sigma_u)_p=x_u({\bu_0})\left(\frac{\partial}{\partial x}\right)_p+y_u({\bu_0})\left(\frac{\partial}{\partial y}\right)_p+t_u({\bu_0})\left(\frac{\partial}{\partial t}\right)_p,\\
&&
(\sigma_v)_p=x_v({\bu_0})\left(\frac{\partial}{\partial x}\right)_p+y_v({\bu_0})\left(\frac{\partial}{\partial y}\right)_p+t_v({\bu_0})\left(\frac{\partial}{\partial t}\right)_p
\end{eqnarray*}
and the normal $N_p$ of $\calS$ at $p$ is the exterior product $(\sigma_u)_p\wedge(\sigma_v)_p$:
$$
N_p=\left|\frac{\partial(y,t)}{\partial(u,v)}\right|_{\bu_0}\left(\frac{\partial}{\partial x}\right)_p+\left|\frac{\partial(t,x)}{\partial(u,v)}\right|_{\bu_0}\left(\frac{\partial}{\partial y}\right)_p+\left|\frac{\partial(x,y)}{\partial(u,v)}\right|_{\bu_0}\left(\frac{\partial}{\partial t}\right)_p.
$$
Here,
$$
\left|\frac{\partial(y,t)}{\partial(u,v)}\right|_{\bu_0}=y_u(\bu_0)t_v(\bu_0)-y_v(\bu_0)t_u(\bu_0)
$$
and similarly for the other determinants.
 The {\it horizontal normal} $N_p^h$ of $\calS$ at $p$ is defined locally by the relation
\begin{equation}\label{eq:hornor}
N_p^h=\left(\left|\frac{\partial(y,t)}{\partial(u,v)}\right|_{\bu_0}+2y(\bu_0)\left|\frac{\partial(x,y)}{\partial(u,v)}\right|_{\bu_0}\right)X_p+\left(\left|\frac{\partial(t,x)}{\partial(u,v)}\right|_{\bu_0}-2x(\bu_0)\left|\frac{\partial(x,y)}{\partial(u,v)}\right|_{\bu_0}\right)Y_p,
\end{equation}
and is an element of the horizontal space ${\rm H}_p(\h)$. 

When $N^h_p$ is not zero we may define the unit horizontal normal $\nu_p^h=N^h_p/\|N^h_p\|$ at $p$. One may show that away from points where the horizontal normal vanishes, a $\cC^1$ horizontal vector field $\nu_{\calS}^h$ of $\calS$ is defined by
$$
(\nu_{\calS}^h)_p=\nu_p^h.
$$
The set of points $p\in\calS$ such that $\|N^h_p\|=0$ is the {\it characteristic locus} $\fC(\calS)$ of $\calS$.

A natural 1-form $\omega_\calS$ is defined on $\calS$ via the inclusion map $\iota:\calS\hookrightarrow\h$; $\omega_\calS$ is just $\iota^*\omega$, where $\omega$ is the contact form of $\h$. The following hold:
\begin{enumerate}
\item[i)] If $\sigma:U\to\R^3$, $\sigma(u,v)=\left(x(u,v),y(u,v),t(u,v)\right)$ is a surface patch of $\calS$, then the following local formula holds:
\begin{equation}\label{eq:omegaS}
\omega_\calS=\sigma^*\omega=(t_u+2xy_u-2yx_u)du+(t_v+2xy_v-2yx_v)dv.
\end{equation}
\item[ii)] The characteristic locus $\fC(\calS)$ is exactly the set of points $p\in\calS$ such that $(\omega_\calS)_p=0$.
\item[iii)] The 1-form $\omega_\calS$ defines an integrable foliation of $\calS$ (with singularities at points of $\fC(\calS)$) by horizontal surface curves. These curves are integral curves of ${\mathbb J}\nu_\calS^h$. Here, $\mathbb{J}$ is the complex operator acting on the horizontal space of $\h$ by the relations $\mathbb{J}X=Y$ and $\mathbb{J}Y=-X$. 
\end{enumerate}

Let $\nu^h=\nu_1X+\nu_2Y$ be the unit horizontal normal vector field of $\calS$; then at a non characteristic point $p$ of $\calS$ the {\it horizontal mean curvature} $H^h(p)$ of $\calS$ at $p$ is just $H^h(p)=X_p\nu_1+Y_p\nu_2$. Suppose that $\sigma:U\to\R^3$, $\sigma(u,v)=(x(u,v),y(u,v),t(u,v))$, is a surface patch of $\calS$ such that $p=\sigma(\bu_0)$ for some $\bu_0=(u_0,v_0)\in U$. Then if
$$
\left|\frac{\partial(x,y)}{\partial(u,v)}\right|_{\bu_0}\neq 0,
$$ we have:
\begin{equation}
H^h(p)=\frac{\left|\frac{\partial(\nu_1,y)}{\partial(u,v)}\right|_{\bu_0}+\left|\frac{\partial(x,\nu_2)}{\partial(u,v)}\right|_{\bu_0}}{\left|\frac{\partial(x,y)}{\partial(u,v)}\right|_{\bu_0}}.
\end{equation}
It can be shown that if $\gamma$ is an integral curve of $\mathbb{J}\nu_\calS^h$ passing from $p$, then $H^h(p)$ is $\kappa_{\bf s}(p')$, the value of the signed curvature of the projection of $\gamma$ on the complex plane at the projection $p'$ of $p$. 

There is a notion of {\it horizontal area} (or perimeter) for $\calS$. For a surface patch $\sigma$ of $\calS$, $\sigma:U\to\R^3$, $\sigma=\sigma(u,v)$, the horizontal area $\calA^h(\sigma)$ of $\sigma$ is
$$
\calA^h(\sigma)=\iint_U\|N^h(u,v)\| dudv,
$$  
where $N^h(u,v)=N^h_{\sigma(u,v)}$. The measure $\|N^h(u,v)\| dudv$ is the local restriction of the 3-dimensional Carnot-Carath\'eodory measure to $\calS$.  

\section{Surfaces of Revolution in the Heisenberg Group}\label{sec:surfrev}
In this section we shall take a brief look in the horizontal geometry of $\cC^2$ surfaces of revolution when these are considered embedded in the Heisenberg group $\h$. In Section \ref{sec:surfrevg} we define the surfaces of revolution we are going to work with; we impose on these surfaces some plausible conditions so that their shape is more or less round. In Section \ref{sec:horgeomrev} we study the aspects of their horizontal geometry; finally, in Section \ref{sec:examples} we give three characteristic examples of surfaces of revolution.
\subsection{Surfaces of Revolution}\label{sec:surfrevg}
We recall the definition of a surface of revolution in the classical sense. Let
\begin{equation}\label{eq:profcurve}
\bp(s)=(f(s),0,g(s)),\quad s\in (a_\bp,b_\bp),
\end{equation}
be a regular, $\cC^2$ curve lying in the $xt$-plane. Regularity here means $\dot f^2(s)+\dot g^2(s)>0$ for $s\in(a_\bp,b_\bp)$. We shall also hereafter consider the following two assumptions for $\bp$:
\begin{equation}\label{A1}
f(s)>0\;\text{for}\; s\in(a_\bp,b_\bp)\;\text{and}\; \lim_{s\to a_\bp^-}f(s)=\lim_{s\to b_\bp^-}f(s)=0.
\end{equation}
That is, there is no $s_0\in(a_\bp,b_\bp)$ such that $\bp(s_0)$ lies on the vertical $t$-axis $\mathcal{V}$. Also,
\begin{equation}\label{A2}
\dot g(s)<0\;\text{for each}\; s\in(a_\bp,b_\bp)\;\text{and}\; \lim_{s\to a_\bp^-}g(s)>0,\;\; \lim_{s\to b_\bp^-}g(s)<0.
\end{equation}
This ensures us that $f$ and $g$ have continuous extensions to $[a_\bp,b_\bp]$ so that $\bp$ may be extended continuously to include the points
$$
N=(0,0,g(a_\bp)),\quad\text{and}\quad S=(0,0,g(b_\bp)),
$$
which we call the {\it poles} of $\bp$. Additionally, $\bp$ is heading continuously downwards from the positive part of the vertical axis $\mathcal{V}$ to the negative part of $\mathcal{V}$. 
\begin{defn}\label{defnsurfrev}
Let $\bp$ be a $\cC^2$ curve as in (\ref{eq:profcurve}). The {\it surface of revolution} $\calS$ with profile curve $\bp$  is the surface obtained by rotating $\bp$ around  the vertical axis $\mathcal{V}$.
\end{defn}
A surface of revolution $\calS$ is a $\cC^2$ surface embedded in $\h^{*}=\C_{*}\times\R$: Let $\sigma$ be the surface patch
\begin{equation}\label{eq:patchrev}
\sigma(s,\phi)=\left(f(s)\cos\phi,f(s)\sin\phi,g(s)\right),\quad (s,\phi)\in (a_\bp,b_\bp)\times (0,2\pi).
\end{equation} 
This covers the whole surface $\calS$ minus the meridian $\bp$. Together with the surface patch $\sigma'$ defined by the same formula as (\ref{eq:patchrev}) but with $(s,\phi)\in(a_\bp,b_\bp)\times (-\pi,\pi)$, we cover the whole surface. Working with the surface patch $\sigma$ as in (\ref{eq:patchrev}) we find that the normal to $\calS$ at $\sigma(s,\phi)$ is
\begin{equation*}
N(s,\phi)=f(s)\left(-\dot g(s)\cos\phi\partial_x-\dot g(s)\sin\phi\partial_y+\dot f(s)\partial_t\right).
\end{equation*}
Its Euclidean norm is $|N(s,\phi)|=f(s)|\dot\bp(s)|\neq 0$; hence we obtain the regularity of $\calS$.
\begin{rem}\label{rem:extrev1}
Let $N$ and $S$ be the poles of $\bp$. The set $\overline{\calS}=\calS\cup\{N,S\}$ is a compact subset of $\h$ which might or might not be a $\cC^2$ surface embedded in $\h$. We note that if $\bp$ is a part of a $\cC^2$ {\it closed} curve which is symmetric with respect to the axis $\mathcal{V}$, then it can be proved that $\overline{\calS}$ is actually a $\cC^2$ surface, which is called an {\it extended} surface of revolution (see Remark 4, p. 77, in \cite{DC}). All surfaces of revolution in Section \ref{sec:examples} are actually extended surfaces of revolution. However. the surfaces of revoloution we consider here are not required to have a $\cC^2$ extension. 
\end{rem}
For clarity, we close this section with the following standard result: The area  of the surface of revolution $\calS$ depends only on its profile curve $\bp$:
\begin{equation*}
\calA(\calS)=2\pi\int_{a_\bp}^{b_\bp} f(s)|\dot\bp(s)|ds,
\end{equation*}
compare to the formula for the horizontal are $\calA^h(\calS)$ in the next section.
\subsection{Horizontal geometry of surfaces of revolution}\label{sec:horgeomrev}
Recal from Section \ref{sec:prel}  the Kor\'anyi map $\alpha:\h^{*}\to \overline{\calL}$,  $\alpha(z,t)=-|z|^2+it$. Consider the $\alpha$-image $\bp^*$ of $\bp$, where $\bp$ is as in (\ref{eq:profcurve}):
\begin{equation*}
\bp^*(s)=-f^2(s)+ig(s),\quad s\in(a_\bp,b_\bp).
\end{equation*}
Note that due to our Conditions (\ref{A1}) and (\ref{A2}) for $\bp$, 
$
\lim_{s\to a_\bp^+}\bp^*(s)
$ is in the positive imaginary axis and 
$
\lim_{s\to b_\bp^-}\bp^*(s)
$ is in the negative imaginary axis of $\overline{\calL}$. Moreover,
$$
|\dot\bp^*(s)|^2=4f^2(s)\dot f^2(s)+\dot g^2(s)\neq 0,
$$
for every $s\in(a_\bp,b_\bp)$.
Using the formulae of Section \ref{sec:surf} we have the following:
\begin{enumerate}
 \item The horizontal normal to $\calS$ is given in the surface patch $\sigma$ by
\begin{equation}\label{eq:hornorsurf}
 N^h(s,\phi)=f(s)\left(-\Im(e^{i\phi}\dot\bp^{*}(s))X_{\sigma(s,\phi)}+\Re(e^{i\phi}\dot\bp^{*}(s))Y_{\sigma(s,\phi)}\right).
\end{equation}
Its horizontal norm is $\|N^h(s,\phi)\|=f(s)|\dot\bp^{*}(s)|$ and therefore the characterictic locus $\mathfrak{C}(\calS)$ of $\calS$ is empty. The unit horizontal normal field $\nu^h_\calS$ of $\calS$ is given in  $\sigma$ by
\begin{equation*}
\nu^h(s,\phi)=-\frac{\Im(e^{i\phi}\dot\bp^{*}(s))}{|\dot\bp^{*}(s)|}X_{\sigma(s,\phi)}+\frac{\Re(e^{i\phi}\dot\bp^{*}(s))}{|\dot\bp^{*}(s)|}Y_{\sigma(s,\phi)}.
\end{equation*}
\item The horizontal area $\calA^h(\calS)=\calA^h(\sigma)$ of $\calS$ depends only on $\bp^*$:
\begin{equation*}\label{eq:horarearev}
\calA^h(\calS)
=2\pi\int_{a_\bp}^{b_\bp}f(s)|\dot\bp^*(s)|ds=2\pi\int_{a_\bp}^{b_\bp}\Re^{1/2}(-\bp^*(s))|\dot\bp^*(s)|ds.
\end{equation*}
\item The induced 1-form of $\calS$ is given in $\sigma$ by
$$
\omega_\calS=\Im(\dot\bp^*(s))ds-2\Re(\bp^*(s))d\phi.
$$
Hence the integral curves of the horizontal flow of $\calS$ passing from $p_0=\sigma(s_0,\phi_0)$ are 
given parametrically by
$$
\phi(s)=\phi_0+\frac{1}{2}\int_{s_0}^s\frac{\Im(\dot\bp^*(u))}{\Re(\bp^*(u))}du.
$$
\item Finally, the horizontal mean curvature of $\calS$ is given by
\begin{equation*}
H^h(s,\phi)=H^h(s)=-\frac{1}{f(s)}\Im\left(\frac{\dot\bp^{*}(s)}{|\dot\bp^{*}(s)|}\right)-\frac{1}{\dot f(s)}\Im\left(\frac{d}{ds}\left(\frac{\dot\bp^{*}(s)}{|\dot\bp^{*}(s)|}\right)\right).
\end{equation*}
\end{enumerate}
\begin{rem}\label{rem:extrev2}
If $\overline{\calS}$ is an extended surface of revolution, see Remark \ref{rem:extrev1}, then the characteristic locus $\mathfrak{C}(\overline{\calS})$ comprises of the poles $N$ and $S$ of $\bp$.
\end{rem}

\subsubsection{Examples}\label{sec:examples} We give three examples of surfaces of revolution which satisfy (\ref{A1}) and (\ref{A2}). These surfaces also share an additional property which we shall use in the next section. 

\medskip

\noindent 1. For $R>0$, the {\it Kor\'anyi sphere} $\calS_\h(R)=\{(z,t)\in\h\;|\;|z|^4+t^2=R^4\}$ minus its two poles $(0,0,\pm R^2)$, is the $cC^2$ surface of revolution with 
profile curve
$$
\bp(\beta)=(R(-\cos\beta)^{1/2},0,R^2\sin\beta),\quad \beta\in(\pi/2,3\pi/2).
$$
Consider the surface patch
$$
\sigma(\beta,\phi)=\left(R(-\cos\beta)^{1/2}e^{i\phi},R^2\sin\beta\right),\quad(\beta,\phi)\in (\pi/2,3\pi/2)\times (0,2\pi)
$$
as well as
$$
\bp^{*}(\beta)=R^2e^{i\beta},\quad \beta\in(\pi/2,3\pi/2).
$$
This is a half circle lying on $\calL$, centred at the origin and with boundary points $\pm iR^2$.
Moreover,
$$
\dot\bp^{*}(\beta)=iR^2e^{i\beta},\quad \beta\in(\pi/2,3\pi/2).
$$
Observe that the curve $\bp^{*}$ is parametrised by its argument.

\medskip

\noindent 2. Our second example is the $\cC^2$ surface of revolution with profile curve
$$
\bp(s)=2R\left(\sin(s/2R),0,R\sin(s/R)-s+\pi R\right),\quad s\in (0,2\pi R).
$$
This is parametrised by the patch
$$
\sigma(s,\phi)=2R\left(\sin(s/2R)e^{i\phi},R\sin(s/R)-s+\pi R\right),\quad(s,\phi)\in (0,2\pi R)\times (0,2\pi).
$$
We have
$$
\bp^{*}(s)=2R^2\left(\cos(s/R)-1+i(\sin(s/R)-s/R+\pi)\right),\quad s\in (0,2\pi R),
$$
and
$$
\dot\bp^{*}(s)=2R\left(-2\sin(s/R)+i(\cos(s/R)-1)\right),\quad s\in (0,2\pi R).
$$
Let $\beta(s)=\arg(\bp^{*}(s))$. Then one shows that
$$
\tan(\beta(s))=\frac{\sin(s/R)-s/R+\pi}{\cos(s/R)-1},
$$
is strictly increasing in $(0,2\pi R)$. Therefore $\bp^{*}$ can be parametrised by its argument. Adding the points $(0,0,\pm\pi R^2)$ to this surface, we obtain the {\it bubble set} $\calB(R)$, see p. 23 in \cite{CDPT}.

\medskip

\noindent 3. For $R>0$, the {\it Carnot-Carath\'eodory sphere} $\calS_{cc}(R)=\{p\in\h\;|\; d_{cc}(0,p)=R\}$ minus its two poles $(0,0,\pm R^2/\pi)$, is a $\cC^2$ surface of revolution with profile curve
\begin{eqnarray*}
 \bp(k)&=&\left(|\bc_k(R)|,0,t_k(R)\right)\\
 &=&\left(\frac{1}{|k|}|1-e^{ikR}|,0,\frac{2}{k}\left(\frac{1}{k}\sin(kR)-R
\right)\right),\quad k\in(-2\pi/R,2\pi/R).
\end{eqnarray*}
It follows that $\sigma:(-2\pi/R,2\pi/R)\times(0,2\pi)\to\R^3$ with
$$
\sigma(k,\phi)=\left(\frac{1}{|k|}|1-e^{ikR}|e^{i\phi},\frac{2}{k}\left(\frac{1}{k}\sin(kR)-R
\right)\right),
$$
is a surface patch for $\calS_{cc}(R)$. Moreover,
$$
\bp^{*}(k)=\frac{2}{k^2}\left(\cos(kR)-1+i\left(\sin(kR)-kR
\right)\right),\quad k\in (-2\pi/R,2\pi/R),
$$
and by a similar argument to that used in the second example one finds that $\bp^{*}$ can also be parametrised by its argument.

\section{Revolution Ring and Revolution Coordinates}\label{sec:revring}
In this section we shall define coordinates for the Heisenberg group which generalise the logarithmic coordinates defined in \cite{P1}, see also \cite{BFP}. With the aid of these coordinates we are going to prove Theorem \ref{thm:main1} in the nest section. 
We first fix our setup: 
\begin{enumerate}
 \item [{(i)}] Throughout this section $\calS$ will be an arbitrary but fixed surface of revolution with profile curve $\bp$ as in (\ref{eq:profcurve}), satisfying (\ref{A1}) and (\ref{A2}). By $\overline{\calS}$ we shall denote the set $\calS\cup\{N,S\}$, where $N,S$ are the poles of $\bp^*$, see Remark \ref{rem:extrev1}. 
 \item [{(ii)}] We shall also impose the following condition: 
If $\beta(s)=\arg(\bp^{*}(s))$ then
 \begin{equation}\label{eq:betacond}
  \dot\beta(s)=\frac{\Im\left({\overline \bp^{*}}(s)\cdot\dot\bp^{*}(s)\right)}{|\bp^{*}(s)|^2}>0,\quad s\in(a_\bp,b_\bp).
 \end{equation}
\end{enumerate}
There is a geometric interpretation of (\ref{eq:betacond}); it simply means that $\beta(s)=\arg(\bp^{*}(s))$ is a strictly increasing function of $s$ and consequently $\bp^{*}$ can be parametrised by $\beta\in(\pi/2,3\pi/2)$. Recall that this is the case for all surfaces in Section \ref{sec:examples}. 
Then
$$
\bp^*=\bp^*(\beta)=-f^2(\beta)+ig(\beta)=|\bp^*(\beta)|e^{i\beta}, \quad \beta\in(\pi/2,3\pi/2),
$$
hence (\ref{eq:betacond}) may be written as
\begin{equation}\label{eq:betacond2}
\Im\left({\overline \bp^{*}}(\beta)\cdot\dot\bp^{*}(\beta)\right)=|\bp^{*}(\beta)|^2>0,\quad \beta\in(\pi/2,3\pi/2).
\end{equation}
It follows that our  $\calS$ 
may be parametrised by the surface patch $\sigma:(\pi/2,3\pi/2)\times(0,2\pi)\to\R^3$, with
$$
\sigma(\beta,\phi)=\left(e^{i\phi}\Re^{1/2}(-\bp^*(\beta)),\;\Im(\bp^*(\beta))\right),\quad (\beta, \phi)\in (\pi/2,3\pi/2)\times(0,2\pi).
$$
We may now represent the set $\overline{\calS}$ as a hypersurface: Recall from Section \ref{sec:heis} the Kor\'anyi gauge $|\cdot|_\h$: $|(z,t)|_\h=|\alpha(z,t)|^{1/2}$ for each $(z,t)\in\h$. Then $\overline{\calS}$ comprises of the points $(z,t)$ in $\h$ such that 
\begin{equation}\label{eq:hypersurface}
\frac{|(z,t)|_\h}{|\bp^*(\arg(\alpha(z,t)))|^{1/2}}=1.
\end{equation}
Obviously, $\calS$ comprises of points  $(z,t)\in\h^*=\C_*\times\R$ satisfying (\ref{eq:hypersurface}) and hence $\calS$ is a $\cC^2$ hypersurface. As for $\overline{S}$, it may or may not be $\cC^2$; both cases do not affect our subsequent discussion.
\begin{defn}\label{defn:rev-ring}
Let $0<a<b$ and let $\calS$ be a surface of revolution such that Conditions (\ref{A1}), (\ref{A2}) and (\ref{eq:betacond}) are satisfied. The {\it revolution ring} $R_{a,b}(\calS)$ is the set 
\begin{equation}\label{eq:ring}
R_{a,b}(\calS)=\left\{(z,t)\in\h \;|\;a<\frac{|(z,t)|_\h}{|\bp^*(\arg(\alpha(z,t)))|^{1/2}}<b\right\}.
\end{equation}
That is,  $R_{a,b}(\calS)$ is the subset of $\h$ which is bounded between the dilated images $D_\alpha(\overline{\calS})$ and $D_b(\overline{\calS})$ of $\overline{\calS}=\calS\cup\{N,S\}$, where $N,S$ are the poles of the profile curve of $\calS$.
\end{defn}

In the remaining of this section we will define revolution coordinates for the set $\h_*=\h\setminus\{(0,0)\}$ and study their properties. Our treatment is similar to that of logarithmic coordinates in \cite{BFP}. Revolution coordinates are defined by the following proposition:
\begin{prop}\label{prop:revcoords}
Let $\calS$ be our fixed surface of revolution as abobe and let $A=\R\times[\pi/2,3\pi/2]$ $\times[0,2\pi]$ with coordinates $(\xi,\beta,\phi)$. We define a map
$\Phi:A\to\h_*$ by 
\begin{equation*}
\Phi(\xi,\beta,\phi)=\left(e^{\xi+i\phi}\Re^{1/2}(-\bp^*(\beta)),\; e^{2\xi}\Im(\bp^{*}(\beta))\right),\quad (\xi,\beta, \phi)\in A.
\end{equation*}
Then $\Phi$ is invertible with inverse $\Phi^{-1}:\h_*\to A$ given by
$$
\Phi^{-1}(z,t)=\left(\frac{1}{2}\log\left(\frac{|\alpha(z,t)|}{|{\bf P}(z,t)|}\right),\;\arg(\alpha(z,t)),\;\arg z\right),\quad(z,t)\in\h_*.
$$
Here, $\alpha(z,t)=-|z|^2+it$ and $
{\bf P}(z,t)=\bp^{*}(\arg(\alpha(z,t))).
$
\end{prop}
\begin{proof}
By our assumptions for the boundary points of $\bp^{*}$, the map $\Phi$ is well defined everywhere on $A$. Moreover, $\Phi$ is invertible. To see this, we set $\Phi(\xi,\beta,\phi)=(z,t)$. Then
$
\arg(\alpha(z,t))=\arg(e^{i\beta})
$
and therefore
$
\beta=\arg(a(z,t)).
$
Also,
$
|a(z,t)|=\left|-|z|^2+it\right|=e^{2\xi}|\bp^{*}(\beta)|=e^{2\xi}|{\bf P}(z,t)|
$. 
Hence
$
\xi=\frac{1}{2}\log\left(|a(z,t)|/|{\bf P}(z,t)|\right).
$
Finally, by definition we have
$
\phi=\arg z
$
and we conclude that $\Phi$ is invertible. 
\end{proof}
The map $\Phi$ is smooth and locally injective on the domain $\tilde\h^*=\R\times(\pi/2,3\pi/2)\times\R$ and $\Phi(\tilde\h^*)=\h^*=\C_*\times\R$. The Jacobian $J_\Phi$ is
$$
J_\Phi(\xi,\beta,\phi)=e^{4\xi}|\bp^{*}(\beta)|^2,\quad (\xi,\beta,\phi)\in\tilde\h^*,
$$
which is strictly positive. Note that $\Phi:\tilde\h^*\to\h^*$ is a smooth covering map. Therefore for each curve $\gamma:[a,b]\to\h^*$ and for each point $(\xi,\beta,\phi)\in\Phi^{-1}(\{\gamma(a)\})$, there exists a unique lifted curve $\tilde\gamma:[a,b]\to\tilde\h^*$ such that $\gamma=\Phi\circ\tilde\gamma$ and $\tilde\gamma(a)=(\xi,\beta,\phi)$. If $\gamma$ is absolutely continuous in the Euclidean sense or $\cC^k$, $k=0,1,\dots$, then the same hold for $\tilde\gamma$ as well. Let $\tilde U\subseteq\tilde\h^*$ be an open set such that $(\tilde U, (\Phi_{\tilde U})^{-1})$ is a local chart for $\h^*$. Straightforward calculations lead to the following local expression for the contact form $\omega$ of $\h$ in revolution coordinates:
\begin{equation*}
\omega=2e^{2\xi}\left(\Im(\bp^*(\beta)) d\xi+\Re(\bp^*(\beta)) d\phi\right)
+e^{2\xi}\Im(\dot\bp^{*}(\beta))d\beta.
\end{equation*}
The next proposition is important for our subsequent discussion; it is analogous to Proposition 10 in \cite{BFP}:
\begin{prop}
A curve $\gamma:[a,b]\to\h^*$ is horizontal if and only if there exists an absolutely continuous curve 
$$
\tilde\gamma:[a,b]\to\tilde\h^*,\quad \tilde\gamma(\tau)=(\xi(\tau),\beta(\tau),\phi(\tau)),
$$
with $\gamma=\Phi\circ\tilde\gamma$ and
\begin{equation}\label{eq:horcurveg}
\dot\phi(\tau)=\tan\left(\beta(\tau)\right)\dot\xi(\tau)+\frac{\Im\left(\dot\bp^{*}(\beta(\tau))\right)}{2\Re\left(\bp^{*}(\beta(\tau))\right)}\dot \beta(\tau).
\end{equation}
If a horizontal curve $\gamma;[a,b]\to\h_*$ satisfies $\gamma(\tau)\in\h^*$ for almost every $\tau\in[a,b]$, then there exists
$$
\tilde\gamma:[a,b]\to\R\times[\pi/2,3\pi/2]\times\R,\quad \gamma(\tau)=(\xi(\tau),\beta(\tau),\phi(\tau)),
$$
with $\tau\mapsto\xi(\tau)$ absolutely continuous, such that for all $\tau\in[a,b]\cap\gamma^{-1}(\h^*)$ we have $\Phi(\tilde\gamma(\tau))=\gamma(\tau)$ and (\ref{eq:horcurveg}) holds almost everywhere. Moreover, if $\rho:\h\to[0,+\infty)$ is any Borel function, then
$$
\int_\gamma\rho ds^h=\int_a^b\rho\left(\Phi(\tilde\gamma(\tau))\right)\frac{e^{\xi(\tau)}}{\Re^{1/2}\left(-\bp^{*}(\beta(\tau))\right)}\left|\bp^{*}(\beta(\tau))\dot\xi(\tau)+\frac{1}{2}\dot\bp^{*}(\beta(\tau))\dot \beta(\tau)\right|d\tau.
$$
\end{prop}
\begin{proof}
If $\tilde\gamma:[a,b]\to\tilde\h^*$ is absolutely continuous and satisfies (\ref{eq:horcurveg}), then we define $\gamma$ to be the absolutely continuous curve $\Phi\circ\tilde\gamma$. Conversely, if $\gamma:[a,b]\to\h^{*}$ is horizontal, we consider the lift $\tilde\gamma$ of $\gamma$ with respect to the covering map $\Phi$. Now suppose that for an almost everywhere differentiable curve $\gamma:[a,b]\to\h_*$ we are given a $\tilde\gamma:[a,b]\to\R\times[\pi/2,3\pi/2]\times\R$ such that $\Phi\circ\tilde\gamma=\gamma$ for all $\tau\in[a,b]\cap\gamma^{-1}(\h^*)$. If $\tau$ is such a point and it is a point of differentiability, then it has a neighbourhood at which we also have $\Phi\circ\tilde\gamma=\gamma$. Writing $\gamma=(\gamma_h,\gamma_3)$ we have
$$
\gamma_h(\tau)=e^{\xi(\tau)+i\phi(\tau)}\Re^{1/2}\left(-\bp^{*}(\beta(\tau))\right).
$$
By differentiating with respect to $\tau$, 
taking absolute values in both sides and using (\ref{eq:horcurveg}) we obtain
\begin{equation}\label{eq:horspeedg}
\left\|\dot\gamma_h(\tau)\right\|= e^{\xi(\tau)}\Re^{-1/2}(-\bp^{*}(\beta(\tau)))\cdot\left|\bp^{*}(\beta(\tau))\dot\xi(\tau)+\frac{1}{2}\dot\bp^{*}(\beta(\tau))\dot \beta(\tau)\right|.
\end{equation}
Thus for such a horizontal curve the formula for the curve integral follows immediately. It only remains to prove the existence of $\gamma$ in the case where $\gamma$ meets the vertical axis $\mathcal{V}$. This is done by arguing exactly as in the last part of the proof of Proposition 10 in \cite{BFP}; we refer the reader there for further details.
\end{proof}
Let $\Omega\subseteq\h$ be a measurable set and let $\tilde\Omega\subseteq\tilde\h^*$ be an open set such that $\Phi(\tilde\Omega)$ coincides with $\Omega$ up to a set of measure zero and such that the restriction of  $\Phi$ to ${\tilde\Omega}$ is invertible. Then a function $\sigma:\Omega\to\R$ is integrable if and only if $(\sigma\circ\Phi)\cdot J_\Phi$ is integrable in $\tilde\Omega$ and in this case:
\begin{equation*}
\iiint_\Omega\sigma(z,t)d\cm^3(z,t)=\iiint_{\tilde\Omega}
\sigma(\Phi(\xi,\beta,\phi))e^{4\xi}|\bp^{*}(\beta)|^2d\cm^3(\xi,\beta,\phi).
\end{equation*}
In particular, for $\Omega=R_{a,b}(\calS)$ as in (\ref{defn:rev-ring}) we have that for every integrable function $\sigma:R_{a,b}(\calS)\to\R$ we have
$$
\iiint_{R_{a,b}(\calS)}\sigma(z,t)d\cm^3(z,t)=\iiint_B\sigma(\Phi(\xi,\beta,\phi)e^{4\xi}|\bp^{*}(\beta)|^2d\cm^3(\xi,\beta,\phi),
$$
where
$$
B=(\log a,\log b)\times(\pi/2,3\pi/2)\times(0,2\pi).
$$

\section{Modulus of Boundary Connecting Curves}\label{sec:mod1}
This section is mainly devoted to the proof of Theorem \ref{thm:main1} below. A final note on the geometric behaviour of quasiradials is in Section \ref{sec:note}. 
\begin{thm}\label{thm:main1}
Let $\calS$ be a surface of revolution with profile curve $\bp$ satisfying Conditions (\ref{A1}), (\ref{A2}). Let $\alpha:\h\to\overline{\calL}$ be the Kor\'anyi map and let $\bp^*=\alpha\circ\bp$ satisfying Condition (\ref{eq:betacond}). Let also  $R_{a,b}(\calS)$, $0<a<b$, be the revolution ring as in (\ref{eq:ring}) and let $\Gamma$ be the family of horizontal curves which lie inside  $R_{a,b}(\calS)$ and connect its two boundary components. Then the following hold for the modulus ${\rm Mod}(\Gamma)$ of $\Gamma$:
\begin{enumerate}
 \item $${\rm Mod}(\Gamma)\ge\pi^2\left(\log(b/a)\right)^{-3}=\iiint_\h\rho_0^4(z,t)d\cm^3(z,t),$$
with $\rho_0$ given by
\begin{equation}\label{eq:rho-0}
\rho_0(z,t)=
\frac{1}{\log(b/a)}\frac{|z|}{(|z|^4+t^2)^{1/2}}\mathcal{X}(z,t),\quad (z,t)\in\h.
\end{equation}
Here, $\mathcal{X}$ is the characteristic function of $R_{a,b}(\calS)$.
 \item  Suppose that $\bp^*$ is parametrised in $(\pi/2,3\pi/2)$ and  denote again by $\bp^*$ the continuous extension of $\bp^*$ to $[\pi/2,3\pi/2]$. 
Then  we have $${\rm Mod}(\Gamma)=\pi^2\left(\log(b/a)\right)^{-3}$$ and $\rho_0$ is an extremal density for $\Gamma$. 
\end{enumerate}
\end{thm}

\medskip

\noindent{\it Proof of Theorem \ref{thm:main1}}.
The proof shall be given in steps. In the first step we prove (1); to do so, we introduce the family $\Gamma_r\subset \Gamma$ of quasiradials and prove that ${\rm Mod}(\Gamma_r)\ge\pi^2(\log(b/a))^{-3}$. In the second step which is the proof of (2), we prove that the density $\rho_0$ as in (\ref{eq:rho-0}) is admissible for the large family $\Gamma$; this concludes the proof.

\smallskip

\noindent{\it Step 1: Quasiradials and a lower bound.} 
We shall denote by $\Gamma_r$ the family of {\it quasiradial curves} lying in $R_{a,b}(\calS)$. This family comprises curves $\gamma_r$ that are defined as follows: For fixed $(\beta,\phi)\in(\pi/2,3\pi/2)$ $\times(0,2\pi)$, let 
$$
\gamma_r(\xi)=\left(e^{\xi+i(\phi+\tan\beta\cdot\xi}\Re^{1/2}(-\bp^*(\beta)),\;e^{2\xi}\Im(\bp^*(\beta))\right),\quad \xi\in[\log a, \log b].
$$
Observe that if $\gamma_r\in\Gamma_r$, then $\gamma_r=\Phi\circ\tilde\gamma_{\beta,\phi}$, where
$$
\tilde\gamma_{\beta,\phi}(\xi)=(\xi,\beta,\phi+\tan\beta\cdot\xi),\quad \xi\in [\log a,\log b].
$$
From (\ref{eq:horcurveg}) it follows that $\Gamma_r$ is a family of horizontal curves. Moreover, if $\gamma_r=(\gamma_{r,h},\gamma_{r,3})\in\Gamma_r$, then from (\ref{eq:horspeedg}) we find
$$
\left\|\dot\gamma_{r,h}(\xi)\right\|=\frac{e^\xi|\bp^{*}(\beta)|}{\Re^{1/2}(-\bp^{*}(\beta))}.
$$
For any admissible $\rho$ for $\Gamma_r$ and for every $\gamma_r=\Phi(\tilde\gamma_{\beta,\phi})$ in $\Gamma_r$ we have:
\begin{eqnarray*}
1\le \int_{\gamma_r}\rho ds^h
&=&\int_{\log a}^{\log b}\rho\left(\gamma_r(\xi)\right)\frac{|\bp^{*}(\beta)|}{\Re^{1/2}(-\bp^{*}(\beta))}e^\xi d\xi\\
&=&\int_{\log a}^{\log b}\rho\left(\gamma_r(\xi)\right)
\frac{|\bp^{*}(\beta)|^{1/2}}{\cos^{1/2}(\beta)}
e^\xi d\xi\\
&\le &\left(\int_{\log a}^{\log b}\rho^4\left(\gamma_r(\xi)\right)e^{4\xi}|\bp^{*}(\beta)|^2 d\xi\right)^{1/4}\cdot \frac{\left(\log(b/a)\right)^{3/4}}{\cos^{1/2}(\beta)},
\end{eqnarray*}
where for the last inequality we have used H\"older's inequality with exponent 4. By taking the last inequality to the fourth power we have
\begin{equation}\label{eq:int}
\left(\log(b/a)\right)^{-3}\cdot \cos^2\beta\le \int_{\log a}^{\log b}\rho^4\left(\gamma_r(\xi)\right)e^{4\xi}|\bp^{*}(\beta)|^2 d\xi.
\end{equation}
Therefore, by integrating (\ref{eq:int}) with respect to $\beta$ and $\phi$ we have
$$
\pi^2\left(\log(b/a)\right)^{-3}\le \iiint_B\rho^4\left(\Phi(\xi,\beta,\phi+\tan\beta\cdot\xi)\right)J_\Phi(\xi,\beta,\phi)d\cm^3(\xi,\beta,\phi),
$$
where $B=(\log a,\log b)\times(\pi/2,3\pi/2)\times(0,2\pi)$. Since the Jacobian determinant of the transformation $(\xi,\beta,\phi)\mapsto(\xi,\beta,\phi+\tan\beta\cdot\xi)$ of $B$ equals to 1, it follows that
\begin{eqnarray*}
\pi^2\left(\log(b/a)\right)^{-3}&\le&\iiint_B\rho^4\left(\Phi(\xi,\beta,\phi)\right)J_\Phi(\xi,\beta,\phi)d\cm^3(\xi,\beta,\phi)\\
&=&\iiint_{R_{a,b}(\calS)}\rho^4(z,t)d\cm^3(z,t).
\end{eqnarray*}
We take the infimum over all admissible $\rho$ of $\Gamma_r$ and use the inequality  ${\rm Mod}(\Gamma_r)\le{\rm Mod}(\Gamma)$  to obtain 
$$
\pi^2\left(\log(b/a)\right)^{-3}\le{\rm Mod}(\Gamma_r)\le{\rm Mod}(\Gamma).
$$
We next verify that the density $\rho_0$ as in (\ref{eq:rho-0}) is extremal. Note that $\rho=\tilde\rho_0\circ\Phi^{-1}$, where
$$
\tilde\rho_0(\xi,\beta,\phi)=\left(\log(b/a)\right)^{-1}e^{-\xi}|\bp^{*}(\beta)|^{-1}\Re^{1/2}\left(-\bp^{*}(\beta)\right)
\mathcal{X}(B).
$$
We have
\begin{eqnarray*}
\iiint_\h\rho_0^4 d\cm^3&=&\iiint_B(\tilde\rho_0)^4(\xi,\beta,\phi)J_\Phi(\xi,\beta,\phi)d\cm^3(\xi,\beta,\phi)
\\
&=&\left(\log(b/a)\right)^{-4}\int_{\log a}^{\log b}\int_{\pi/2}^{3\pi/2}\int_0^{2\pi}\frac{e^{4\xi}|\bp^{*}(\beta)|^2\Re^{2}\left(\bp^{*}(\beta)\right)}{e^{4\xi}|\bp^{*}(\beta)|^4}d\phi d\beta d\xi\\
&=&2\pi\left(\log(b/a)\right)^{-3}\int_{\pi/2}^{3\pi/2}\cos^2\beta d\beta\\
&=&\pi^2\left(\log(b/a)\right)^{-3}.
\end{eqnarray*}
\noindent {\it Step 2: Admissibility of $\rho_0$.}
To prove (2) we consider the density $\rho_0$ as above. 
The density $\rho_0$ is admissible for the family $\Gamma$: Pick any $\gamma=\Phi\circ\tilde\gamma\in\Gamma$ and suppose that
$$
\tilde\gamma(\tau)=\left(\xi(\tau),\beta(\tau),\phi(\tau)\right),
$$
where $\tau$ is the horizontal arc length parameter running through an interval $[0,\ell]$. We now write $\bp^*(\beta)=r(\beta)e^{i\beta}$, $\beta\in[\pi/2,3\pi/2]$, and we may assume that $r(\beta(\ell))\ge r(\beta(0))$; if the opposite inequality holds, we change the arc length parameter from $\tau$ to $\ell-\tau$.  Under this normalisation we may also assume that $\xi(0)=\log a$ and $\xi(\ell)=\log b$. Since $\gamma$ is of unit horizontal speed we have from Equation (\ref{eq:horspeedg}) that
\begin{eqnarray*}
e^{-\xi(\tau)}\Re^{1/2}\left(\bp^{*}(\beta(\tau))\right)&=&\left|\bp^{*}(\beta(\tau))\dot\xi(\tau)+\frac{1}{2}\frac{d\bp^{*}}{d\beta}(\beta(\tau))\dot \beta(\tau)\right|\\
&=&\left|r(\beta(\tau))e^{i\beta(\tau)}\dot\xi(\tau)+\frac{1}{2}\left(\frac{dr}{d\beta}(\beta(\tau))e^{i\beta(\tau)}+ir(\beta(\tau))e^{i\beta(\tau)}\right)\dot\beta(\tau)\right|\\
&=&\left|r(\beta(\tau))\dot\xi(\tau)+\frac{1}{2}\left(\frac{dr}{d\beta}(\beta(\tau))+ir(\beta(\tau))\right)\dot\beta(\tau)\right|.
\end{eqnarray*}
Therefore
\begin{eqnarray*}
\int_{\gamma}\rho_0ds^h&=&\left(\log(b/a)\right)^{-1}\int_0^\ell\left|\dot\xi(\tau)+\frac{1}{2}
\left(\frac{(dr/d\beta)(\beta(\tau))}{r(\beta(\tau))}+i\right)\dot \beta(\tau)\right|d\tau\\
&\ge &\left(\log(b/a)\right)^{-1}\left(\int_0^\ell\dot\xi(\tau)d\tau+\frac{1}{2}\int_0^\ell\frac{(dr/d\beta)(\beta(\tau))}{r(\beta(\tau))}\dot \beta(\tau)d\tau\right).
\end{eqnarray*}
But
$
\int_0^\ell\dot\xi(\tau)d\tau=\log(b/a),
$
and
\begin{equation*}
\frac{1}{2}\int_0^\ell\frac{(dr/d\beta)(\beta(\tau))}{r(\beta(\tau))}\dot \beta(\tau)d\tau
=\frac{1}{2}\log\frac{r\left(\beta(\ell)\right)}{r\left(\beta(0)\right)}\ge 0.
\end{equation*}
Thus
$$
\int_{\gamma}\rho_0 ds^h\ge 1, 
$$
and $\rho_0$ is admissible for $\Gamma$. 
The proof of Theorem \ref{thm:main1} is complete.
\qed

\subsection{A note on quasiradials}\label{sec:note}
The revolution ring $R_{a,b}(\calS)$ is foliated by dilated images of $\calS$. The following proposition justifies the name of the elements of $\Gamma_r$, that is, they are curves whose horizontal tangent at each point $p$ of the revolution ring $R_{a,b}(\calS)$ is almost parallel to the horizontal normal at $p$ to the leaf $\calS_p$ containing $p$. Moreover, quasiradials are radials if and only if $\calS$ is a Kor\'anyi sphere.
\begin{prop}\label{prop:Gr}
Let $p=\Phi(\xi,\beta,\phi)$ in $R_{a,b}({\bf \calS})$. Then the horizontal tangent $(\dot\gamma_{r,h})_p$ to the quasiradial $\gamma_r$ passing from $p$ and the horizontal normal $N^h_p$ of the surface of revolution inside  $R_{a,b}({\bf \calS})$ containing $p$ make an angle
$$
\theta_p=\frac{\pi}{2}-\arg(\dot\bp^{*}(\beta))+\beta,
$$
at $p$. Therefore $(\dot\gamma_{r,h})_p$ and $N^h_{p}$ are never orthogonal and they are parallel if and only if $R_{a,b}(\calS)$ is a Kor\'anyi ring.
\end{prop}
\begin{proof}
Let $\tilde\gamma_{\beta,\phi-\tan\beta\cdot\xi}(\tau)=(\tau,\beta,\phi+\tan\beta\cdot(\tau-\xi))$, $\tau\in[\log a,\log b]$, $\gamma_r=\Phi\circ\tilde\gamma_{\beta,\phi}$ and $\gamma_r(\xi)=p$. The horizontal tangent to $\gamma_r$ at $p$ is
$$
(\dot\gamma_{r,h})_p=-e^\xi|\bp^*(\beta)|^{1/2}\Re^{-1/2}(-\bp^*(\beta))\cdot\left(\cos(\phi+\beta)X_p+\sin(\phi+\beta)Y_p\right).
$$
 Consider the leaf $\calS_\xi=D_{e^{\xi}}(\calS)$ with surface patch $\sigma_\xi(\cdot,\cdot)=\Phi(\xi,\cdot,\cdot)$. Then $p\in\calS_\xi$ and from (\ref{eq:hornorsurf}) we find that its horizontal normal vector of $\calS_\xi$ at $p$ is
 $$
 N_{p}^h=|\dot\bp^{*}(\beta)|\Re^{1/2}(-\bp^*(\beta))\cdot\left(-\sin(\phi+\arg(\dot\bp^{*}(\beta))X_p+\cos(\phi+\arg(\dot\bp^{*}(\beta))Y_p\right). 
 $$
 We have
 $$
 \langle (\dot\gamma_{r,h})_p, N_{p}^h\rangle=e^\xi|\bp^*(\beta)|^{1/2}|\dot\bp^*(\beta)|\sin\left(\arg(\dot\bp^{*}(\beta)-\beta)\right)
 \neq 0,
 $$
 according to (\ref{eq:betacond}); hence $(\dot\gamma_{r,h})_p$ and $N_{p}^h$ are not orthogonal. Since 
$$
\|(\dot\gamma_{r,h})_p\|=e^\xi|\bp^*(\beta)|^{1/2}\Re^{-1/2}(-\bp^*(\beta))
\;\;\text{and}\;\;
\|N_{p}^h\|=\Re^{1/2}(-\bp^*(\beta)|\dot\bp^{*}(\beta)|,
$$
we obtain
$
\cos\theta_p=\sin\left(\arg(\dot\bp^{*}(\beta))-\beta\right)
$
and our formula for $\theta_p$ follows. 

Finally, $(\dot\gamma_{r,h})_p$ and $N_{p}^h$ are parallel if and only if
$
\arg(\dot\bp^{*}(\beta))=\beta\pm\pi/2.
$
But then $
|\bp^*(\beta)|=\lambda,
$
for some $\lambda>0$, that is,  $\bp^{*}(\beta)=(\alpha\circ\bp)(\beta)$ is a circle in the left half plane $\calL$ centred at the origin. This happens if and only if $\bp$ is a curve lying on a Kor\'anyi sphere and the proof is complete. 
 \end{proof}

 \end{document}